\documentclass{conm-p-l}
\usepackage{mathrsfs}
\usepackage{amsmath, amsthm, amssymb, enumerate, mathtools}

\usepackage{color}

\numberwithin{equation}{section}

\makeatletter
\@namedef{subjclassname@2020}{\textup{2020} Mathematics Subject Classification}
\makeatother

\title[$1$-Meix\-ner random vectors]{Explicit descriptions of all 
$\,1$-Meix\-ner random vectors}

\author[\rm F\nnh.\,Catrina, A.\,Derdzinski, A.\,I.\,Stan]{Florin
Catrina$^{1}$, Andrzej Derdzinski$^{2}$, and Aurel I.\,Stan$^{3}$}
\address{$^{1}$ Department of Mathematics and Computer Science, St.\,John’s
University, Queens, NY 11439, USA}
\email{catrinaf@stjohns.edu}
\address{$^{2}$\nnh\ Department\nnh\ of\nnh\ Mathematics,\nnh\ The\nnh\
Ohio\nnh\ State\nnh\ University,\nnh\ Columbus,\nnh\ OH\nnh\ 43210,\nnh\ USA}
\email{andrzej@math.ohio-state.edu}
\address{$^{3}$ Department of Mathematics, The
Ohio State University, Marion, OH 43302, USA}
\email{stan.7@osu.edu}


\def\bbC{{\mathchoice {\setbox0=\hbox{$\displaystyle\rm C$}\hbox{\hbox
to0pt{\kern0.4\wd0\vrule height0.9\ht0\hss}\box0}}
{\setbox0=\hbox{$\textstyle\rm C$}\hbox{\hbox
to0pt{\kern0.4\wd0\vrule height0.9\ht0\hss}\box0}}
{\setbox0=\hbox{$\scriptstyle\rm C$}\hbox{\hbox
to0pt{\kern0.4\wd0\vrule height0.9\ht0\hss}\box0}}
{\setbox0=\hbox{$\scriptscriptstyle\rm C$}\hbox{\hbox
to0pt{\kern0.4\wd0\vrule height0.9\ht0\hss}\box0}}}}
\def\sy{\sigma}
\def\lg{\langle}
\def\rg{\rangle}
\def\lr{\lg\,,\rg}
\def\ew{V}
\def\ys{E}
\def\sm{\mathrm{V}\hskip-2pt_m\w}
\def\smo{\mathrm{V}\hskip-2pt_{m-1}\w}
\def\sd{\mathrm{V}\hskip-2.4pt_{d-2}\w}
\def\sdt{\mathrm{V}\hskip-2.4pt_{d-3}\w}
\def\sz{\mathrm{V}\hskip-1.8pt_0\w}
\def\uo{U\hskip-1.9pt_1\w}
\def\ud{U\hskip-2.9pt_d\w}
\def\ui{U\hskip-2.2pt_i\w}
\def\yy{Y\hskip-1.4pt}
\def\yo{Y\hskip-3.2pt_1\w}
\def\yd{Y\hskip-4.3pt_d\w}
\def\yi{Y\hskip-3.2pt_i\w}
\def\yj{Y\hskip-4.2pt_j\w}

\def\bbH{\mathrm{I\!H}}
\def\bbK{\mathrm{I\!K}}
\def\bbN{\mathrm{I\!N}}
\def\bbR{\mathrm{I\!R}}

\def\rto{\bbR\nh^2}

\newcommand{\w}{^{\phantom i}}
\newcommand{\nnh}{\hskip-1pt}

\def\hs{h\hskip-.6pt_6\w}

\usepackage{color}

\def\cwedge{\bigcirc\kern-1.07em\wedge\ }
\newcommand{\hyp}{\hskip.5pt\vbox
{\hbox{\vrule width2.5ptheight0.5ptdepth0pt}\vskip2pt}\hskip.5pt}
\def\vt{{\tau\hskip-4.55pt\iota\hskip.6pt}}
\def\evt{{\tau\hskip-3.55pt\iota\hskip.6pt}} 
\def\hs{\hskip.7pt}
\def\hh{\hskip.4pt}
\def\nh{\hskip-.7pt}
\def\hn{\hskip-.4pt}
\def\w{^{\phantom i}}

\def\ve{\varepsilon}

\newtheorem{thm}{\bf Theorem}[section]
\newtheorem{lem}[thm]{\bf Lemma}
\newtheorem{prop}[thm]{\bf Proposition}
\newtheorem{cor}[thm]{\bf Corollary}

\theoremstyle{definition}

\theoremstyle{remark}
\newtheorem{rem}[thm]{\noindent\bf Remark}

\newtheorem*{pf}{\noindent\it Proof}

\numberwithin{equation}{section}

\begin{document}
\begin{abstract}
The 1-Meix\-ner random vectors consist of $d$ real-valued random 
variables, having finite moments of all orders,
for which the commutators between their joint semi-quantum operators and the multiplication operators that they generate are linear combinations of these multiplication operators and the identity operator. This property 
of random variables imposes a system of first-order partial differential 
equations on their joint La\-place transform on a neighborhood 
of $0$ in a Euclidean $d$-space. 
We provide 
explicit descriptions of all solutions of the system 
just mentioned in every dimension $d$.
\end{abstract}

\maketitle

\setcounter{thm}{0}
\renewcommand{\thethm}{\Alph{thm}}
\section{Introduction}
\setcounter{equation}{0}
The goal of this paper  is to find all solutions of the equation
\begin{equation}\label{nph}
[\nabla\nh\varphi](x)\,=\,\alpha(x,[\nabla\nh\varphi](x))\,
+\,\varphi(x)\hh x\
\end{equation}
imposed on the unknown $\,C^2$ function $\,x\mapsto\varphi(x)\,$ defined on
a sufficiently small connected neighborhood of $\,0\,$ in $\,\ew\nh$, where
the following objects are involved:
\begin{equation}\label{tbg}
\begin{array}{l}
\mathrm{a\hs\ Euclidean\hs\ space\hs\ }\,\hs\ew\hs\,\mathrm{\ of\hs\
dimension\ }\hs\,d\,\ge\hs1\hs\mathrm{\hs\ with\hs\ the\hs\ in}\hyp\\
\mathrm{ner\hh\ product\hh\ }\hh\lr\hh\mathrm{\hh\
and\hh\ a\hh\ bi\-lin\-e\-ar\hs\ 
mapping\hs\ }\,\alpha\nnh:\nnh\ew\nnh\nh\times\nh\ew\nnh\to\hs\ew\\
\mathrm{such\ that\ }\hs\lg\alpha(x,y),z\rg\hs\mathrm{\ is\
totally\ symmetric\ in\ }\hs x,y,z\in\ew\nnh\nnh,\\
\mathrm{where\nh\ we\nh\ also\nh\ write\nh\ 
}\,\alpha_x\w\nh
=\hh\alpha(x,\,\cdot\,)\nnh:\nnh\ew\nnh\to\nh\ew\nnh\mathrm{,\nnh\
for\nh\ any\nnh\ 
}\,x\in\nh\ew\nnh\nnh.
\end{array}
\end{equation}
As shown in \cite{stan-catrina} --- see Proposition~\ref{omreq} 
below -- {\it the
La\-place transform\/ $\,\varphi\,$ of every nondegenerate\/  
$\,1$-Meix\-ner random vector\/ $\,\yy=(\yo\dots,\yd)$, normalized so that}
\begin{equation}\label{nrm}
\yo,\dots,\yd\,\mathrm{\ are\ }\,L\nnh^2\nnh\hyp\mathrm{or\-tho\-nor\-mal\ and\
}\,L\nnh^2\nnh\hyp\mathrm{or\-thog\-o\-nal\ to\ constants,}
\end{equation}
{\it is a solution of equation\/} (\ref{nph}), {\it and\/} $\,\varphi(0)=1$.

The solutions $\,\varphi\,$ of (\ref{nph}) with $\,\varphi(0)=1\,$ naturally
form two classes, described in Theorem~\ref{first}. The present paper also 
establishes a partial converse of the above  italicized statement, by
exhibiting the joint probability distributions of the components of  
such normalized $\,1$-Meix\-ner random vectors corresponding to 
some suitably rescaled versions of solutions belonging to one
the two classes just mentioned.

It is worth pointing out, which we do in Proposition~\ref{unqra}, that
\begin{equation}\label{ivp}
\mathrm{any\ solution\ }\,\varphi\,\mathrm{\ to\ (\ref{nph})\ is\
uniquely\ determined\ by\ }\,\varphi(0),
\end{equation}
as well as real-an\-a\-lyt\-ic, while, unless $\,\varphi\,$ vanishes 
identically,
\begin{equation}\label{unq}
\varphi\,\mathrm{\ uniquely\ determines\ both\ }\,\alpha\,\mathrm{\ and\ the\ inner\ product\ }\,\lr.
\end{equation}
Whenever $\,\varphi\,$ is not identically zero, (\ref{nph}) with 
$\,\sy=\log\hh|\varphi|\,$ reads
\begin{equation}\label{rwr}
[\nabla\nh\sy](x)\,=\,\alpha(x,[\nabla\nh\sy](x))\,+\,x\,\mathrm{\ for\
}\,x\,\mathrm{\ near\ }\,0\,\mathrm{\ in\ }\,V\nnh.
\end{equation}
Section~\ref{tm} describes several trivial ways of modifying given examples 
of $\,\varphi\,$ with (\ref{nph}), for suitable $\,\ew,d,\alpha\,$ and $\,\lr$,
to obtain new examples with the same property.
The first two are the 
$\,\mathrm{O}\hh(\hn d\hh)${\it-equiv\-a\-lence\/} and 
{\it orthogonal di\-rect-sum construction}.
Then, there is {\it rescaling\/}: with
$\,\ew\nh,d\,$ and $\,\lr\,$ in (\ref{tbg}) unchanged, we use any fixed 
$\,c\in\bbR\smallsetminus\{0\}\,$ to replace $\,\alpha$, and a solution
$\,\sy=\log\hh|\varphi|\,$ to (\ref{rwr}), by
\begin{equation}\label{rsc}
\widetilde\alpha=c\hs\alpha\,\mathrm{\ and\ the\ function\
}\,x\mapsto\widetilde\sy(x)=c^{-\nh2}\nh\sy(c\hh x).
\end{equation}
Here is our first main result.
\begin{thm}\label{first}All solutions\/ $\,\varphi\hs$ to\/
{\rm(\ref{nph})}, with their associated data\/ $\,\ew\hskip-2pt,\alpha,\lr$,
arise via the trivial modifications just mentioned from two classes of
examples. In both classes we have
\begin{equation}\label{fcs}
\varphi=ae^\sy\,\mathrm{\ for\ an\ arbitrary\ }\,a\in\bbR\,\mathrm{\
 and\ some\ specific\ function\ }\,\sy\nh.
\end{equation}
In the first class, $\,\ew\nh=\bbR\hn^d\nh$, with any $\,d\ge1$, and\/
$\,\lr\,$ is the
standard inner product, while the components\/  
$\,\alpha\hn_{i\hn jk}\w=\lg\alpha(e\hn_i\w,e\nh_j\w),e\hn_k\w\rg\hs$ of\/  
$\,\alpha\,$ in the standard or\-tho\-nor\-mal basis
$\,e\hn_1\w,\dots,e\nh_d\w$ and the function\/ $\,\sy\hs$ are given by
\begin{equation}\label{cmp}
\begin{array}{l}
\alpha_{ii\hh d}\w=\alpha_{id\hh i}\w=\alpha_{d\hh ii}\w=1\,\mathrm{\ for\
 all\
}\,i\mathrm{,\ and\ }\,\alpha\hn_{i\hn jk}\w=0\hs\mathrm{\ otherwise,}\\
2\sy(x\hn_1\w,\dots,x\nh_d\w)\,=\,-\nh2x\nh_d\w\hs 
-\,\log\hh|(1-x\nh_d\w)^2\nh-\textstyle{\sum}\hs_{i\hh<\hn d}\w\,x_i^2|.
\end{array}
\end{equation}
The second class of examples uses a vector sub\-space\/ $\,\ew\hskip-2.3pt$,
closed under 
the anti\-com\-muta\-tor operation\/ $\,(y,z)\mapsto yz+zy$, of the 
vector space of all self-ad\-joint endo\-mor\-phisms of some given Euclidean 
space\/ $\,\ys\nh$, with $\,\lr,\,\,\alpha\,$ and\/ $\,\sy\,$ defined,
for\/ $\,y,z\in\ew$ and\/ $\,x\,$ near\/ $\,0\,$ in\/ $\,\ew\nh$, by
\begin{equation}\label{scl}
\begin{array}{l}
2\alpha(y,z)\,=\,yz+zy,\quad2\lg y,z\rg\,=\,\mathrm{tr}\,yz,\\
2\hh\sy(x)\,=\,-\mathrm{tr}\,x\,-\,\log\hn|\hh\mathrm{det}\hs(\mathrm{Id}-x)|.
\end{array}
\end{equation}  
\end{thm}
Note that all $\,\varphi\,$ in (\ref{fcs}) -- (\ref{cmp}) are 
$\,\mathrm{O}\hh(\hn d-1\hh)$-in\-var\-i\-ant:
\begin{equation}\label{inv}
\varphi=e^{-\xi}\hh[(1-\xi)^2\nnh-s^2]^{-\nnh1/2}\nh\mathrm{,\ for\ 
}\,\xi=x\nh_d\w\mathrm{\ and\ }\,s\ge0\,\mathrm{\ with\  
}\,s^2\nnh=\textstyle{\sum}\hs_{i\hh<\hn d}\w\,x_i^2.
\end{equation}
\noindent\begin{rem}\label{spcas}Important special cases of the
second class of
examples arise when
\begin{enumerate}
\item[{\rm(i)}] $\ew\hs$ is the full space of all self-ad\-joint
endo\-mor\-phisms of
$\,\ys\nnh$, or
\item[{\rm(ii)}] $\ys\hs$ is the underlying real Euclidean space of a
complex or quater\-nion\-ic in\-ner-prod\-uct space, and $\,\ew\nh$ consists
of those self-ad\-joint endo\-mor\-phisms which are also
$\,\bbC$-lin\-e\-ar or, respectively, $\,\bbH$-lin\-e\-ar, with
$\,\mathrm{tr}\,$ and $\,\mathrm{det}\,$ still standing for the
{\it real\/} trace and determinant; or, 
\item[{\rm(iii)}] $\ew\,$ is the space of $\,3\nh\times\nh3\,$ octonion
matrices, self-ad\-joint in a suitable sense, and viewed as 
$\,24\times\nh24\,$ real symmetric matrices, so that  
$\,\ys\nh=\bbR\hn^{24}\nnh$.
\end{enumerate}
\end{rem}
See Section~\ref{pf} for more details.

Let $\,\alpha\,$ be as in (\ref{tbg}). For $\,x\in\ew\nnh$ we set
$\,K[x]=\mathrm{span}\hs\{\alpha_x^nx:n\in\bbN\cup\{0\}\}$, 
where $\,\bbN\,$ is the set of positive integers. 
Thus, since $\,\alpha_x^0=\mathrm{Id}\,$ even for $\,\alpha_x\w=0$,
\begin{equation}\label{hve}
K[x]\,=\,\hs\mathrm{span}\,[\{x\}\hs\cup\hs\{\alpha_x^nx:n\in\bbN\}].
\end{equation}
Our two further results characterize the case where (\ref{nph}) is
nontrivially solvable.
\begin{thm}\label{equiv}Given\/ $\,\alpha\hs$ with\/
{\rm(\ref{tbg})}, three conditions are mutually equivalent\/{\rm:}
\begin{enumerate}
\item[{\rm(i)}] some\/ $\,\varphi$, not identically zero,
satisfies\/ {\rm(\ref{nph})},
\item[{\rm(ii)}] $\alpha_x\w=\alpha(x,\,\cdot\,)\,$ and\/ 
$\,\alpha(\alpha_x^nx,\,\cdot\,)\,$ commute for every\/ $\,x\in\ew\hskip-2pt$
and\/ $\,n\in\bbN$,
\item[{\rm(iii)}] whenever\/ $\,x\in\ew\hskip-2.3pt$, all\/ $\,\alpha_y\w$,
where\/ $\,y\in K[x]$, commute with one another.
\end{enumerate}
\end{thm}
\noindent\begin{cor}\label{jdalg}For\/ $\,\alpha\hs$ as in\/ {\rm(\ref{tbg})},
equation\/ {\rm(\ref{nph})} has a nontrivial
solution\/ $\,\varphi$
if and only if\/ $\,\ew\hskip-2pt$ with the multiplication $\,\alpha$
constitutes a Jor\-dan algebra.
\end{cor}
We use the two-sid\-ed, or bilateral, La\-place transform
(no restriction of integration to a ``positive'' quadrant or
half-line), with the same sign convention as 
Rot\-haus \cite{rothaus}. Thus, a Bo\-rel measure $\,\mu\,$
on a Euclidean $\,d$-space $\,\ew\,$ has the La\-place transform
$\,\varphi\,$ defined on a suitable (necessarily convex)
sub\-set of $\,\ew$ by
\begin{equation}\label{lap}
\varphi(x)
=\nnh\displaystyle{\int}_{\hskip-4.5pt\ew}\w
e^{\lg x,y\rg}\hh d\mu\nh_y\w.
\end{equation}
When $\,\mu\,$ is absolutely continuous with respect to the Le\-besgue
measure, we treat it as a function, by identifying it with its
Ra\-don–Ni\-ko\-dym derivative.

In the next theorem, the $\,\mathrm{O}\hh(\hn d-1\hh)$-in\-var\-i\-ance of 
$\,f\,$ reflects the same property of its La\-place transform, in the form of
a rescaled version of (\ref{inv}).

The symbol $\,u\,$ always denotes 
the unit-step (Heaviside) function.
\begin{thm}\label{frcls}In the first class of examples of 
Theorem\/~{\rm\ref{first}}, the rescaled version of
$\,\varphi\,$ with\/ {\rm(\ref{fcs})} -- {\rm(\ref{cmp})}, for every
dimension\/ $\,d\ge2\,$ and any\/
$\,c\,$ in\/ {\rm(\ref{rsc})} such that\/ $\,(d+2)\hh c^2\in(0,1]$, 
has the inverse La\-place transform\/ $\,f\,$ which 
depends on
\begin{equation}\label{yey}
y=(y\hn_1\w,\dots,y\hn_d\w)\,\mathrm{\ only\ through\
}\,\eta=y\nh_d\w\mathrm{\ and\ }\,r\ge0\,\mathrm{\
given\ by\ }\,r^2\nnh=\textstyle{\sum}\hs_{i\hh<\hn d}\w\,y_i^2.
\end{equation}
Specifically, with a constant\/ $\,\kappa\,$ depending on\/ $\,d\,$ and\/
$\,c$,
\begin{equation}\label{slt}
f\hs
=\,\kappa\hs e^{-\eta/c}\hh[(1+c\hh\eta)^2\nnh-\hs c^2\nh r^2]\hh^{p/2}
\hs u(1+c\hh\eta-|c|r)\mathrm{,\ where\ }\,p=-d+1/c^2\nh.
\end{equation}
\end{thm}
To prove Theorem~\ref{frcls}, in Sect.\,\ref{pl}, we use the uniqueness 
of a solution\ $\,\varphi\,$ to (\ref{nph}) with $\,\varphi(0)=1$,
mentioned in (\ref{ivp}), which
relieves us of a rather tedious task: 
evaluating the La\-place transform of $\,f\,$ directly
from (\ref{lap}).

For $\,d\in\{1,2,3\}$, the inverse La\-place transforms
of the original functions $\,\varphi$ with (\ref{fcs}) -- (\ref{cmp})
and $\,\varphi(0)=1$, not subjected to any rescaling, were described in
\cite{stan-catrina}. 
We provide concise
versions of those three descriptions in the Appendix.

The paper is organized as follows. We devote the final three sections to  
reviewing briefly the
joint quantum and sem\-i-quan\-tum operators generated by a finite family of
random variables having finite moments of all orders, and 
stating the result from \cite{stan-catrina} which shows that equation 
(\ref{nph}) holds for the joint La\-place transforms of the components of all 
suitably normalized nondegenerate $\,1$-Meix\-ner random
vectors. That material is preceded by comments concerning intrinsic
aspects of (\ref{nph}), unrelated to its quan\-tum-prob\-a\-bil\-i\-ty
interpretation, and proofs of the results listed above. To be specific, 
Sections\,~\ref{xp} through \ref{pd} contain a quick verification of 
(\ref{nph}) for the examples described in Theorem~\ref{first},
a description of the trivial modifications of solutions to (\ref{nph}), and
a proof of Theorem~\ref{equiv}. To prove the uniqueness assertion of 
Theorem~\ref{first}, in Section~\ref{pf}, we first trivially
derive Corollary~\ref{jdalg} from Theorem~\ref{equiv}. We then point out 
that, due to (\ref{tbg}), any Jor\-dan algebra arising in 
Corollary~\ref{jdalg} is of the Euclidean type, and hence formally real. 
In their seminal 1934 paper \cite{jordan-von-neumann-wigner},
Pascual Jor\-dan, John von Neu\-mann and Eugene Wigner classified 
all formally real Jor\-dan algebras, and the uniqueness claim in 
Theorem~\ref{first} is an immediate consequence of their 
classification. 
Finally, in Section~\ref{pl}, we prove Theorem~\ref{frcls}.

The present paper leaves unanswered two questions, to be addressed in the 
future. One concerns a possible converse of \cite[Lemma 6]{stan-catrina}, 
that is, of Proposition~\ref{omreq}, also stated in the lines surrounding 
(\ref{nrm}). More precisely, let an $\,\bbR\hn^d\nh$-val\-ued random variable 
$\,\yy=(\yo\dots,\yd)\,$ on a probability space 
have finite moments of all orders 
and satisfy the normalizing condition 
(\ref{nrm}). For the La\-place transform $\,\varphi\,$  
of the resulting push-for\-ward Bo\-rel measure 
on $\,\bbR\hn^d\nh$,
{\it does equation\/ {\rm(\ref{nph})} with any suitable data\/ {\rm(\ref{tbg})}
guarantee that\/ $\,\yy\hs$ is a  $\,1$-Meix\-ner random vector}? The
other question is: {\it Can the\/ $\,1$-Meix\-ner condition be characterized
by certain group invariance, for instance, in the spirit of\/ \cite{letac-89} 
and\/} \cite{letac-94}?

\renewcommand{\thethm}{\thesection.\arabic{thm}}
\renewcommand{\theprop}{\thesection.\arabic{prop}}
\section{The examples described in Theorem~\ref{first}}\label{xp}
\setcounter{equation}{0}
For $\,\alpha\,$ appearing in (\ref{cmp}), equation (\ref{rwr}) obviously
amounts to
\begin{equation}\label{omx}
(1-x\nh_d\w)\hh\sy\nnh_i\w=(1+\sy\hskip-2.3pt_d\w)x\nh_i\w\,\mathrm{\ if\
}\,i<d\,\mathrm{\ and\ }\,(1-x\nh_d\w)\hh\sy\hskip-2.3pt_d\w
=x\nh_d\w+\textstyle{\sum}\hs_{i\hh<\hn d}\w\,x\nh_i\w\sy\nnh_i\w,
\end{equation}
with subscripts applied to $\,\sy$ denoting partial differentiations.
The relations (\ref{omx}) are easily verified to hold if $\,\sy$ is
defined by (\ref{cmp}).

When $\,\alpha\,$ is given by (\ref{scl}), equation (\ref{rwr}) takes the form
\begin{equation}\label{rds}
2\hh[\hh d_y\w\sy](x)=[\hh d_{xy+yx}\w\sy](x)+\mathrm{tr}\,xy\,\mathrm{\ 
for\ any\ }\hs y\hn\in\nh\ew\hs\mathrm{\ and\ }\hs x\hn\in\nh\ew\hs\mathrm{\
close\ to\ }\hs0. 
\end{equation}  
In the case of a real vector space $\,\ys\,$ of any finite dimension, 
$\,\theta:\mathrm{GL}\hh(\ys)\to\bbR\,$ defined by
$\,\theta(x)=\log\hn|\hh\mathrm{det}\,x|\,$ satisfies, whenever 
$\,x\in\mathrm{GL}\hh(\ys)\,$ and $\,y\in\mathfrak{gl}\hh(\ys)$, the
following well-known relation (see Remark~\ref{dettr} below):
\begin{equation}\label{dvt}
  [\hh d_y\theta\hh](x)\,\hs=\,\,\mathrm{tr}\,x^{-\nnh1}\nh y.
\end{equation}
With $\,\sy\,$ as in (\ref{scl}), using (\ref{dvt}) we get 
$\,2\hh[\hh d_z\w\sy](x)=-\mathrm{tr}\,z
+\mathrm{tr}\hs(\mathrm{Id}-x)^{-\nnh1}\nnh z\,$ 
for $\,x\,$ near $\,0\,$ and any $\,z\in\ew$. Applied to $\,z=y\,$ and
$\,z=xy+yx$, this gives
\[
\begin{array}{l}
2\hh[\hh d_y\w\sy](x)-[\hh d_{xy+yx}\w\sy](x)-\mathrm{tr}\,xy
=[-\mathrm{tr}\,y+\mathrm{tr}\hs(\mathrm{Id}-x)^{-\nnh1}\nnh y]\\
\hskip156pt+\,[\mathrm{tr}\,xy-\mathrm{tr}\hs(\mathrm{Id}-x)^{-\nnh1}\nnh xy]
-\mathrm{tr}\,xy.
\end{array}
\]
Canceling, on the right-hand side,
the two occurrences of $\,\mathrm{tr}\,xy\,$ with opposite signs,
we see that the right-hand side equals the inner product of $\,2y\,$ and
\[
-\mathrm{Id}\,+\,(\mathrm{Id}-x)^{-\nnh1}\hs-\,(\mathrm{Id}-x)^{-\nnh1}\nnh x\,
=\,-\mathrm{Id}\,+\,(\mathrm{Id}-x)^{-\nnh1}\nh(\mathrm{Id}-x)\,=\,0,
\]
proving (\ref{rds}) and, consequently, (\ref{rwr}).
\begin{rem}\label{dettr}For any $\,z\in\mathfrak{gl}\hh(\ys)$,
applying $\,d/dt\,$ at $\,t=0\,$ to
$\,\mathrm{det}\,\hs\mathrm{exp}\hskip1.2ptt\hn z$ we
get $\,\mathrm{tr}\,z$, since
$\,\mathrm{tr}:\mathfrak{gl}\hh(\ys)\to\bbK\,$ is the Lie-al\-ge\-bra 
homo\-mor\-phism induced by the Lie-group
homo\-mor\-phism
$\,\mathrm{det}:\mathrm{GL}\hh(\ys)\to\bbK\smallsetminus\{0\}$.
(Equivalently, $\,\mathrm{tr}\,$ is the differential of $\,\mathrm{det}\,$
at $\,\mathrm{Id}$, since $\,-z\,$ has, in dimension $\,m$,
the characteristic polynomial
$\,t\mapsto(-\nnh1)^m[\hs t^m\nnh+(\mathrm{tr}\,z)\hh t^{m-1}]$ plus
low\-er-or\-der terms,
and so $\,\mathrm{det}\hs(\mathrm{Id}+tz)$ equals
$\,1+(\mathrm{tr}\,z)\hh t\,$ plus high\-er-or\-der terms.) Using
the curve $\,t\mapsto x\,\mathrm{exp}\,tx^{-\nnh1}\nh y$, which for
$\,t=0\,$ arrives at $\,x\,$ with the velocity $\,y$, 
one easily obtains (\ref{dvt}).
\end{rem}

\section{Trivial modifications}\label{tm}
\setcounter{equation}{0}
We begin by justifying the claims (\ref{ivp}) and (\ref{unq}) made in
the Introduction. First, concerning the operator
$\,\mathrm{Id}-\alpha_x\w:\ew\to\ew\hskip-3pt$, depending as in (\ref{tbg}) on
$\,x\in\ew\hskip-3pt$,
\begin{equation}\label{ivt}
\mathrm{Id}-\alpha_x\w\mathrm{\ is\ invertible\ for\ }\,x=0\mathrm{,\ and\
hence\ for\ all\ }\,x\,\mathrm{\ near\ }\,0,
\end{equation}
so that (\ref{nph}) is in normal form (solved for the derivatives), and
(\ref{ivp}) follows.
\begin{prop}\label{unqra}If\/ {\rm(\ref{rwr})} has a\/ $\,C^2$ solution\/ 
$\,\sy$ on a neighborhood of\/ $\,0\,$ in\/ $\,\ew\hskip-2.5pt$,
then\/ $\,\sy$ is
real-an\-a\-lyt\-ic, uniquely determined by\/ $\,\alpha$ up to an 
additive constant, and, explicitly, $\,\sy(x)=\lg x,\gamma\nnh_x\w x\rg$,
where, with the notation of\/ {\rm(\ref{tbg})}, $\,\gamma\nnh_x\w$ is the
linear operator\/ $\,\ew\to\,\ew\,$ given by
\begin{equation}\label{xpl}
\gamma\nnh_x\w=\nnh\displaystyle{\int}_{\hskip-4.5pt0}^1t\,[\hh\mathrm{Id}
-t\alpha_x\w]^{-\nnh1}\hn dt\,\mathrm{\ or,\ equivalently,\ 
}\,\gamma\nnh_x\w
=\hn\mathop{\displaystyle{\sum}}\limits_{n\hs=\,0}^\infty
(n+2)^{-\nnh1}\nh\alpha_x^n.
\end{equation}
Furthermore, $\,\sy\hn$ uniquely determines both\/ $\,\alpha$ and the 
inner product\/ $\,\lr$, while\/ $\,\sy$ has at\/ $\,0\,$ a strict local 
minimum, with the Hess\-i\-an equal to\/ $\,\lr$.
\end{prop}
\begin{pf}We have (\ref{ivt}), which turns
(\ref{rwr}) into a formula for the gradient of $\,\sy\nh$. Thus, if
a solution $\,\sy\hs$ exists, it may be assumed to vanish at $\,0$, and then
$\,\sy(x)\,$ equals, for $\,x\,$ close to $\,0$, the line integral of the
gradient along the curve $\,[\hs0,1]\ni t\mapsto tx$, which proves the
first claim.

The final clause is in turn obvious since, for all $\,x,y,z\in\ew\nh$,
\[
[\hh d_x\w\sy](0)=0,\quad
[\hh d_x\w d_y\w\sy](0)=\lg x,y\rg,\quad
[\hh d_x\w d_y\w d_z\w\sy](0)=2\lg\alpha(x,y),z\rg,
\]
as one sees differentiating (\ref{rwr}) twice. 
\qed
\end{pf}

There are at least three
trivial ways of constructing new examples of $\,\alpha\,$ with
(\ref{tbg}) out of old ones, while keeping equation (\ref{rwr}) solvable.

The first one may be termed {\it rescaling}. With
$\,\ew\nh,d\,$ and $\,\lr\,$ in (\ref{tbg}) unchanged, we use any fixed 
$\,c\in\bbR\smallsetminus\{0\}\,$ to replace $\,\alpha$, and a solution
$\,\sy\,$ to (\ref{rwr}), by (\ref{rsc}). Secondly,
we say that two examples of $\,\alpha\,$ with (\ref{tbg}), in the same
dimension $\,d$, are $\,\mathrm{O}\hh(\hn d\hh)${\it-equiv\-a\-lent\/}
if one has
the same components as the other in suitably chosen or\-tho\-nor\-mal bases
of the two underlying Euclidean spaces.

Finally, we have the {\it orthogonal di\-rect-sum\/} construction,
where $\,\ew\nh$ and $\,\alpha\,$ in (\ref{tbg}) are the orthogonal direct sums
of $\,\ew'\nnh,\alpha'$ and $\,\ew''\nnh,\alpha''$ satisfying the analogs of 
(\ref{tbg}) in dimensions $\,d\hh'$ and $\,d\hh''\nh$, so that
\[
\alpha(x'\nh,x'')=0,\quad\alpha(x',y\hh')=\alpha'\nh(x'\nh,y\hh'),\quad
\alpha(x''\nh,y\hh'')=\alpha''\nh(x''\nh,y\hh'')
\]
whenever $\,x'\nh,y\hh'\in\ew'\nh$ and $\,x''\nh,y\hh''\in\ew''\nnh$. 
The solvability of (\ref{rwr}) for $\,\sy\,$
amounts to the same for both summands, and then it follows that 
$\,\sy(x'\nh\nnh+x'')=\sy'\nh(x')\hn+\sy''\nh(x'')$, where
$\,\sy'$ and $\,\sy''$ are the solutions for
$\,\ew'\nnh,\alpha'$ and $\,\ew''\nnh,\alpha''\nh$. To see this, one just 
fixes $\,x'$ or $\,x''$ in $\,\sy(x'\nnh+x'')$.

Following \cite[p.\,37]{jordan-von-neumann-wigner}, we call $\,\alpha\,$
{\it irreducible\/} when it is not such a direct sum, and {\it reducible\/}
otherwise. Thus, when $\,\alpha\,$ is reducible, our problem 
is reduced to understanding what happens in lower dimensions.

In addition, (\ref{rwr}) remains unaffected, since so does the gradient of
$\,\sy$, when $\,\sy$ and the inner product $\,\lr\,$ are replaced
by $\,r\hn\sy\,$ and $\,r\lr$, with any real $\,r>0$.

\section{Proof of Theorem~\ref{equiv}}\label{pd}
\setcounter{equation}{0}
\begin{lem}\label{imply}For\/ $\,\alpha\,$ as in\/ {\rm(\ref{tbg})}, the
existence of a solution\/ $\,\sy$ to\/ {\rm(\ref{rwr})} is 
equivalent to the condition\/ {\rm(ii)} in Theorem\/~{\rm\ref{equiv}}.
\end{lem}
\begin{proof}A solution $\,\sy\,$ to (\ref{rwr}), defined near $\,0\,$ in $\,\ew\nh$,
exists if and only if
\begin{equation}\label{ifo}
\begin{array}{l}
\mathrm{for\nnh\ }\,x\,\mathrm{\nnh\ near\nnh\ }\,0\,\mathrm{\nnh\ in\
}\,\ew\nnh\,\mathrm{\ the\nnh\ operator\nnh\ }\,\alpha(x,\,\cdot\,)\\
\mathrm{commutes\hs\ with\  
}\,\hs\alpha([\hh\mathrm{Id}-\alpha(x,\,\cdot\,)]^{-\nnh1}\nh x,\,\cdot\,).
\end{array}
\end{equation}
In fact,
let us set $\,w(x)=[\hh\mathrm{Id}-\alpha(x,\,\cdot\,)]^{-\nnh1}\nh x$,
so that the requirement in question is for the vector field $\,w\,$ to be,
locally, a gradient or, in other words, for the directional derivative 
operator sending $\,y\in\ew\hh$ to $\,d_y\w w\in\ew\hh$ to be 
self-ad\-joint at each point.
Let $\,A(x)=\mathrm{Id}-\alpha(x,\,\cdot\,)\,$ and
$\,B(x)=[A(x)]^{-\nnh1}\nh$, so that $\,w(x)=[B(x)]\hh x\,$ and, if we
suppress the dependence on $\,x\,$ in the notation, 
$\,d_y\w B=B\hh[\hh\alpha(y,\,\cdot\,)]B$. Thus,
$\,[\hh d_y\w w](x)\,=\,[B(x)][\hh\alpha(y,\,\cdot\,)][B(x)]x\,
+\,[B(x)]y\,$ which, as $\,w(x)=[B(x)]\hh x$, reads
\[
[\hh d_y\w w](x)\,=\,[B(x)][\hh\alpha(y,w(x))\,+\,y].
\]
On the right-hand side we have, evaluated on $\,y$,
a composition of two self-ad\-joint 
operators, $\,[B(x)]\,$ and $\,\alpha(w(x),\,\cdot\,)\,+\,\mathrm{Id}$, 
which is itself self-ad\-joint if and only if they commute. Ignoring the
occurrences of $\,\mathrm{Id}$, we see that the condition is
the commutation of 
$\,\alpha(w(x),\,\cdot\,)$ with $\,B(x)$, that is, with its inverse
$\,A(x)=\mathrm{Id}-\alpha(x,\,\cdot\,)$, and consequently with
$\,\alpha(x,\,\cdot\,)$. This is precisely (\ref{ifo}), as required.

With $\,\alpha_x\w=\alpha(x,\,\cdot\,)$, we obviously have, for $\,x\,$ near
$\,0\,$ in $\,\ew\,$ and $\,t\in[\hs0,1]$,
\begin{equation}\label{srs}
[\hh\mathrm{Id}-t\alpha(x,\,\cdot\,)]^{-\nnh1}\hs
=\,\mathop{\displaystyle{\sum}}\limits_{n\hs=\,0}^\infty t^n\nh\alpha_x^n.
\end{equation}
Due to (\ref{ifo}) which, in view of (\ref{rsc}),
remains unaffected when one multiplies 
$\,\alpha\,$ by any $\,t\ne0$, it now follows
that (i) in Theorem~\ref{equiv} is
equivalent to (ii).
\end{proof}
\begin{lem}\label{diago}Suppose that\/ $\,\alpha\,$ satisfies\/ 
{\rm(\ref{tbg})} and\/ {\rm(\ref{rwr})} has a solution\/ $\,\sy\nh$.
\begin{enumerate}
\item[{\rm(i)}] There exists a dense open set\/ $\,Z\subseteq\ew\,$ formed
by\/ $\hs\,x\hh$ such that $\,\alpha_x\w=\alpha(x,\,\cdot\,)$ has the maximum
possible number of distinct eigen\-values and, at the same time, 
$\,K[x]\,$ in\/ {\rm(\ref{hve})} is of dimension\/
$\,\mathrm{max}\hs\{\dim K[y]:y\in\ew\}$. 
\item[{\rm(ii)}] Given\/ $\,x\in Z$, any or\-tho\-nor\-mal basis\/ 
$\,e\hn_1\w,\dots,e\nh_d\w$ of\/ $\,\ew\nh$, di\-ag\-o\-nal\-iz\-ing
$\,\alpha_x\w$, must also di\-ag\-o\-nal\-ize all\/ $\,\alpha_y\w$ with\/
$\,y\in K[x]$.
\end{enumerate}
\end{lem}
\begin{proof}The assertion (i) is immediate, since $\,Z\,$ is defined to be
the intersection of two specific dense open sets.

Any fixed or\-tho\-nor\-mal basis 
di\-ag\-o\-nal\-iz\-ing $\,\alpha_x\w$ di\-ag\-o\-nal\-izes
$\,\alpha_y\w$ for $\,y\in K[x]$ sufficiently close to $\,x$. In fact,
as $\,x\in K[x]$, such $\,\alpha_y\w$ commute with $\,\alpha_x\w$  due to 
Lemma~\ref{imply}, so that they leave the 
eigen\-spaces of $\,\alpha_x\w$ invariant, and
have to act on each eigen\-space as a multiple 
of the identity, or else $\,\alpha_y\w$ would have more distinct
eigen\-values than $\,\alpha_x\w$ does. 
An obvious real-an\-a\-lyt\-ic\-i\-ty argument yields the same conclusion
for all $\,y\in K[x]$, whether or not close to $\,x$, proving (ii).
\end{proof}
\begin{proof}[{\it Proof of Theorem\/~{\rm\ref{equiv}}}]
Assume (i). 
Choosing $\,x\in Z\,$ and a basis $\,e\hn_1\w,\dots,e\nh_d\w$ as in 
part (ii) of Lemma~\ref{diago}, we see that
$\,\alpha(\alpha_x^nx,\,\cdot\,)$, for all integers $\,n\ge0$,
being simultaneously di\-ag\-o\-nal\-iz\-able, 
commute with one another. Since $\,Z\subseteq\ew\,$ is open and dense, the 
same holds, in view of 
real-an\-a\-lyt\-ic\-i\-ty, for all $\,x\in\ew\nh$, so that (iii) follows.
As (iii) trivially implies (ii), Lemma~\ref{imply} completes the proof.
\end{proof}

\section{Proof of uniqueness in Theorem~\ref{first}}\label{pf}
\setcounter{equation}{0}
By a {\it Jor\-dan algebra\/} \cite[p.\,24]{faraut-koranyi}
one means a real or complex vector space
$\,\ew\hh$ with a symmetric bi\-lin\-ear multiplication
$\,\alpha:\ew\nnh\times\ew\to\ew\hh$ such that, for every
$\,x\in\ew\hskip-2pt$, the
$\,\alpha$-mul\-ti\-pli\-ca\-tions by $\,x\,$ and by the
$\,\alpha$\hh-\hh square of $\,x\,$ commute:
\[
\left[\hs\alpha(x,\,\cdot\,),\,\alpha(\alpha(x,x),\,\cdot\,)\right]\,=\,0.
\]
Corollary~\ref{jdalg} is thus a trivial consequence of (ii) in
Theorem~\ref{equiv} for $\,n=1\,$ and the well-known fact
\cite[p.\,27]{faraut-koranyi} that (ii) in
Theorem~\ref{equiv}, as a property of $\,\alpha$,
holds whenever $\,\alpha\,$ is a Jor\-dan-al\-ge\-bra multiplication.

One refers to a real Jor\-dan algebra $\,\ew\,$ as {\it Euclidean\/}
\cite[p.\,42]{faraut-koranyi} if it satisfies 
(\ref{tbg}) for some inner product $\,\lr\,$ in $\,\ew$. This is known
\cite[pp.\,42,\,154]{faraut-koranyi} to be the case if and only if the algebra 
is {\it formally real\/} in the sense that $\,\alpha(x,x)+\alpha(y,y)\ne0,$
unless $\,x=y=0$. Thus, being obviously Euclidean, 
\begin{equation}\label{fre}
\mathrm{the\ Jor\-dan\ algebras\ arising\ in\ Corollary~\ref{jdalg}\ 
are\ all\ formally\ real.}
\end{equation}
A Jor\-dan algebra is said to be {\it irreducible\/}
\cite[p.\,37]{jordan-von-neumann-wigner} if it is not 
a direct sum of low\-er-di\-men\-sion\-al ones. 
Irreducible formally real Jor\-dan algebras were classified by 
Jor\-dan, von Neu\-mann and Wigner
\cite[p.\,63, Fundamental Theorem 2]{jordan-von-neumann-wigner}, and their 
list of types consists of the first class of examples of $\,\alpha\hs$
in Theorem~\ref{first} (which they list under $\,\gamma=1\,$ and
$\,\gamma=2$) and the three cases of Remark~\ref{spcas} (referred to as
$\,\gamma=3$ and $\,\gamma\ge4$). See also \cite[Chapter V]{faraut-koranyi}.

Due to (\ref{fre}), this proves the uniqueness clause of Theorem~\ref{first}.

As pointed out by Jor\-dan, von Neu\-mann and Wigner 
\cite[p.\,63]{jordan-von-neumann-wigner}, in our first class of examples of
$\,\alpha$, the cases $\,d=3\,$ and $\,d=4\,$ coincide with the case 
(i) of Remark~\ref{spcas}
and, respectively, (ii) for the scalar field $\,\bbC\,$ in
Remark~\ref{spcas} where, in both (i) and (ii), $\ys\hs$ is of real/com\-plex
dimension two.

The examples in our Theorem~\ref{first} were designed with the purpose of
maximally simplifying the presentation, and result in abundant choices 
of Jor\-dan algebras which are reducible. For instance, $\,\ew\nh$ in
the second class of examples could be the algebra generated by
$\,\{\mathrm{Id},x\}$, with a single self-ad\-joint endo\-mor\-phism
$\,x\,$ of $\,\ys\nnh$.

\section{Proof of Theorem~\ref{frcls}}\label{pl}
\setcounter{equation}{0}
For any $\,c\in\bbR\smallsetminus\{0\}\,$ and $\,\lambda,\ve\in(0,\infty)$,
\begin{equation}\label{red}
\displaystyle{\int}_{\hskip-5pt0}^\infty \Pi(r)\,dr\,<\,\infty\mathrm{,\
where\ 
}\,\Pi(r)=\nnh\displaystyle{\int}_{\hskip-5pt{r/|c|}}^\infty
e^{-\lambda t}\hh t\hs^\ve\hn dt.
\end{equation}
In fact, $\,\Pi(r)=\lambda\nh^{-\ve-1}\hn\Gamma(\ve+1,\lambda r/|c|)$, where
$\,\Gamma(\delta,q)=\int_q^\infty e^{-t}\hh t\hs^{\delta-1}\hn dt\,$  
is the upper incomplete gamma function. Now (\ref{red}) follows since
$\,\Gamma(\delta,0)=\Gamma(\delta)\,$ is finite while, 
according to \cite[p.\,70]{olver},
$\,\Gamma(\delta,q)\le e^{-q}\hh q\hh^\delta\nnh/(q-\delta+1)\,$ whenever
$\,q>\delta-1>0$.

If $\,c\in\bbR\smallsetminus\{0\}\,$ and
$\,x=(x\hn_1\w,\dots,x\nh_d\w)\,$ lies 
in the open half-cone given by
\begin{equation}\label{hfc}
1-c\hh\xi>|c|s\mathrm{,\ for\ }\,\xi=x\nh_d\w\mathrm{\ and\ 
 }\,s\ge0\,\mathrm{\ such\ that\ 
}\,s^2\nnh=\textstyle{\sum}\hs_{i\hh<\hn d}\w\,x_i^2,
\end{equation}
then, for $\,\eta=y\nh_d\w$, any $\,q\in(0,\infty)$, and $\,r\ge0\,$ with 
$\,r^2\nnh=\textstyle{\sum}\hs_{i\hh<\hn d}\w\,y_i^2$,
as in (\ref{yey}),
\begin{equation}\label{fin}
\displaystyle{\int}_{\hskip-5.5pt\bbR\hn^d}\w\,
e^{\lg x,y\rg -\eta/c}\hh(1+c\hh\eta)^q\hn u(1+c\hh\eta-|c|r)\,dy\,<\,\infty.
\end{equation}
To prove (\ref{fin}), we begin by setting $\,t=(1+c\hh\eta)/c^2\nh$. We thus
have
\begin{equation}\label{ipm}
\lg x,y\rg\,-\,\eta/c\,\le\,-\lambda t\,+\,(1-c\hh\xi)/c^2\nh\mathrm{\ \ with\ 
}\,\lambda=1-c\hh\xi-|c|s.
\end{equation}
Namely, the Schwarz inequality gives $\,\lg x,y\rg\le sr+\xi\eta$, and
so $\,c^2\nh(\lg x,y\rg-\eta/c)\le|c|s|c|r+(c\hh\xi-1)c\hh\eta$. As
$\,|c|r\le1+c\hh\eta=c^2\nh t\,$ in view of the Heaviside factor in 
(\ref{fin}), replacing $\,c\hh\eta\,$ in the term $\,(c\hh\xi-1)c\hh\eta\,$
with $\,c^2\nh t-1$, we get (\ref{ipm}).

Let, again, $\,t=(1+c\hh\eta)/c^2\nh$, and 
$\,(y\hn_1\w,\dots,y\hn_{d-1}\w)=rz$, where 
$\,r\in[\hs0,\infty)\,$ and $\,z$ ranges over the unit sphere
$\,S\hh^{d-2}$ around $\,0\,$ in $\,\bbR\hn^{d-1}\nh$. Now
(\ref{fin}) reads
\[
|c|^{2q}\nh\nnh\displaystyle{\int}_{\hskip-5pt0}^\infty
\hskip-3.5pt\displaystyle{\int}_{\hskip-5ptr/|c|}^\infty
\hh\displaystyle{\int}_{\hskip-5.5ptS^{d-2}}\w
e^{\lg x,y\rg-\eta/c}\hh t\hs^q\hh r^{d-2}\nh dz\,dt\,d\hh r\,<\,\infty.
\]
Since $\,r^{d-2}\le|c|^{d-2}\hh t^{d-2}$ wherever 
$\,t\ge r/|c|$, 
with $\,\sd$ denoting the 
$\,(d-2)$-di\-men\-sion\-al volume of $\,S\hh^{d-2}\nh$,
we see that, by (\ref{ipm}), the last integral does not exceed
\[
|c|^{2q+d-2}\hs e^{1/c^2}\nnh e^{-\xi/c}\,\sd
\nnh\displaystyle{\int}_{\hskip-5pt0}^\infty
\hskip-3.5pt\displaystyle{\int}_{\hskip-5ptr/|c|}^\infty
e^{-\lambda t}\hh t\hs^{q+d-2}\hn dt\,d\hh r,
\]
and (\ref{fin}) follows from (\ref{red}) for $\,\ve=q+d-2$, as (\ref{hfc}) and
(\ref{ipm}) give $\,\lambda>0$.

For $\,f\,$ defined by (\ref{slt}), and $\,i=1,\dots,d$,
we have 
\begin{equation}\label{ltr}
\mathrm{a)}\hskip6pt\partial\nh_i\w\widehat f\nh=\widehat{\,y\hn_i\w f\,}\nh,
\qquad\mathrm{b)}\hskip6ptx\hn_i\w\widehat f\hs
=\,-\widehat{\partial\nh_i\w f}\nh.
\end{equation}
Here $\,\widehat{\hskip5pt}\,$ stands for the La\-place transform,
the transforms of functions of $\,(y\hn_1\w,\dots,y\hn_d\w)\,$ 
depend on $\,(x\hn_1\w,\dots,x\nh_d\w)$, while $\,\partial\nh_i\w$ equals
$\,\partial/\partial\hs x\hn_i\w$ in (\ref{ltr}-a) and 
$\,\partial/\partial\hs y\hn_i\w$ in (\ref{ltr}-b). Also, the domains
of all the transforms appearing in (\ref{ltr}) contain 
the open half-cone (\ref{hfc}).

The last claim, about the domains of the transforms of 
$\,f\nh,\,y\hn_i\w f\,$ and, respectively,
$\,\partial\nh_i\w f\nh$, follows from (\ref{fin})
and the obvious inequality $\,p\ge2\,$ in (\ref{slt}) since,
\begin{equation}\label{psi}
\mathrm{setting\ }\,\,\psi\,=\,(1+c\hh\eta)^2\nnh-\hs c^2\nh r^2\nh\mathrm{,\ 
one\ has\ }\,\,0\le\psi^{\delta/2}\nh\le(1+c\hh\eta)^\delta 
\end{equation}
whenever $\,\delta\in[\hs0,\infty)$,
due to the Heaviside factor in (\ref{fin}). More
precisely, for $\,f$ we use (\ref{psi}) and
(\ref{fin}) with $\,q=p\ge2$. If $\,i<d$,
the definition of $\,r\,$ preceding (\ref{fin}) and the Heaviside factor in 
(\ref{fin}) give $\,|c||y\hn_i\w|\le|c|r\le1+c\hh\eta$. Thus, as   
$\,|\hh\partial\nh_i\w\psi\hh^{p/2}|
=p\hh c^2\hn|y\hn_i\w|\psi^{(p-2)/2}$ and
$\,|y\nh_i\w\psi\hh^{p/2}|=|y\nh_i\w|\,|\psi\hh^{p/2}|$, we get, again from
(\ref{psi}), 
$\,|\hh\partial\nh_i\w\psi\hh^{p/2}|\le p|c|(1+c\hh\eta)^{p-1}$ and
$\,|c||y\nh_i\w\psi\hh^{p/2}|\le(1+c\hh\eta)^{p+1}\nnh$. Now 
(\ref{fin}) with $\,q\,$ equal to $\,p-1\ge1\,$ or, respectively, to $\,p+1$,
yields the required conclusion about the domains. When $\,i=d$, 
similarly,  
$\,|c|\hs|\hh\partial\nnh_d\w(e^{-\eta/c}\hn\psi\hh^{p/2})|
\le e^{-\eta/c}\hn[\psi\hh^{p/2}\nnh+p\hh c^2\hn\psi^{(p-2)/2}\nh(1+c\hh\eta)]
\le e^{-\eta/c}\hn[(1+c\hh\eta)^p\nnh+p\hh c^2\hn(1+c\hh\eta)^{p+1}]\,$ and, 
with $\,|c||y\nh_d\w|=|c\eta|\le(1+c\hh\eta)+1$, one sees that 
$\,|c|\hs|y\nh_d\w\psi\hh^{p/2}|
\le|c|\hh[(1+c\hh\eta)^{p+1}\nnh+(1+c\hh\eta)^p]$.

We obtain (\ref{ltr}-a) from (\ref{fin}) combined with the Le\-besgue
dominated convergence theorem. The integration by parts leading to 
(\ref{ltr}-b) is also permitted: the boundary terms in question
all vanish. For $\,i<d\,$ this is clear as $\,\psi\,$ given by (\ref{psi}) 
equals $\,0\,$ on the boundary half-cone $\,1+c\hh\eta=|c|r$, and hence so
does $\,f\,$ due to positivity of the exponent $\,p\ge2\,$ in
(\ref{slt}). Finally, when $\,i=d$, every line parallel to the
$\,d\hh$th coordinate direction intersects the solid half-cone
$\,1+c\hh\eta\ge|c|r\,$ along a half-line $\,L$, and $\,f=0\,$ at both
``endpoints'' of $\,L\,$ (one of which is infinite). 

On the other hand, (\ref{slt}) easily implies that, with
$\,\partial\nh_i\w=\partial/\partial\hs y\hn_i\w$,
\begin{equation}\label{amt}
\begin{array}{ll}
(1+c\hh y\nh_d)\hs\partial\nh_i\w f\hs+c\hh y\nh_i\w\partial\nnh_d\w f\hs
=\,-y\nh_i\w f\,\mathrm{\ \,for\,\ all\ }\,\,i<d,\\
(1+c\hh y\nh_d)\hs\partial\nnh_d\w f\hs
+\,c\hh\textstyle{\sum}\hs_{i\hh<\hn d}\w\,
y\nh_i\w\partial\nh_i\w f\hs
=\,-(cd\,+\,y\nh_d\w)f\nh
\end{array}
\end{equation}
which, as a consequence of (\ref{ltr}), amounts to
\begin{equation}\label{rea}
\begin{array}{ll}
(1-c\hh x\nh_d)\hs\partial\nh_i\w\widehat f=
(\widehat f\nh+c\hs\partial\nnh_d\w\widehat f)x\nh_i\w\,\mathrm{\ if\ }\,i<d,\\
(1-c\hh x\nh_d)\hs\partial\nnh_d\w\widehat f\,
=\,x\nh_d\w\widehat f\,
+\,c\hs\textstyle{\sum}\hs_{i\hh<\hn d}\w\,
x\nh_i\w\partial\nh_i\w\widehat f\nh.
\end{array}
\end{equation}
The system (\ref{rea}) is nothing else than (\ref{omx}) 
with $\,\alpha\,$ replaced by
$\,\widetilde\alpha=c\hs\alpha$, and treated as a condition imposed on
$\,\widehat f\nh=e^\sy\nh$, rather than $\,\sy\nh$. 
The uniqueness -- mentioned in (\ref{ivp}) -- 
of a solution\ $\,\varphi\,$ to (\ref{nph}) with $\,\varphi(0)=1$,
applied to $\,\widetilde\alpha=c\hs\alpha\,$ rather than
$\,\alpha$, now proves Theorem~\ref{frcls}, for a suitably chosen
constant $\,\kappa$.

\section{Background and summary}\label{bg}
\setcounter{equation}{0}
In dimension one, the classic Meix\-ner class consists of the random 
variables having Sze\-g\H{o}\hs-Ja\-co\-bi parameters $\,\alpha_n\w$ and 
$\,\omega_n\w$ that are, respectively, linear in $\,n\,$ and quadratic
in $\,n\,$ with zero constant term. Up to rescaling and translations, this 
class can be split into six types: 
Gauss\-i\-an, Pois\-son, Gam\-ma, negative binomial, two\hh-pa\-ra\-me\-ter
hyperbolic, and binomial random variables. Attempts to extend this definition 
to higher dimensions used to run into a difficulty, due to the absence of a
reasonable replacement of the Sze\-g\H{o}\hs-Ja\-co\-bi parameters. 
However, in the past two decades, such a replacement has been proposed, in
the form of joint quantum operators: creation, preservation, and annihilation, 
generated by the random variables $\,\yo,\dots,\yd$, which are the
components of a random vector. Splitting each preservation operator into two
halves, one attached to the corresponding creation operator, the other to
the annihilation operator, one obtains the sem\-i-quan\-tum operators:
sem\-i-cre\-a\-tion and sem\-i-an\-ni\-hi\-la\-tion operators. As pointed out 
in \cite{popa-stan-15}, the Meix\-ner random variables can be characterized
by the fact that a nested double commutator involving the
sem\-i-an\-ni\-hi\-la\-tion operator and the multiplication operator generated
by the random variable $\,\yy\,$ is a linear combination of the
sem\-i-quan\-tum operators. Moreover, within this class, the Meix\-ner
sub\-class consisting, up to a rescaling and translation, of the Gauss\-i\-an
and Gam\-ma random variables, have the property that the 
(simple) commutator of the sem\-i-an\-ni\-hi\-la\-tion operator 
$\,\,U$ and the multiplication operator $\,\yy\,$ equals a 
linear combination of $\,\yy\,$ and the identity operator. These
observations allow us to both recover the classic Meix\-ner class and 
define a hierarchy of Meix\-ner random variables, in an operatorial manner,
involving the number of nested commutators, the joint
sem\-i-an\-ni\-hi\-la\-tion operator $\,\,\ui$, and the multiplication 
operators $\,\yi$, $\,i=1,\dots,d$, all of which are used in defining them.
Those using only one commutator in their definition are called
$\,1$-Meix\-ner random vectors, those with two nested commutators -- 
referred to as $\,2$-Meix\-ner, etc. It was shown in
\cite{popa-stan-15} that every nondegenerate two-di\-men\-sion\-al 
$\,1$-Meix\-ner random vector can be reduced, via an invertible affine
transformation, to a random vector whose components are independent
Gam\-ma or Gauss\-i\-an random variables. An analogous result of
\cite{stan-catrina} states that a nondegenerate three-di\-men\-sion\-al
$\,1$-Meix\-ner random vector can be similarly reduced 
either to a random vector with components as above, or to one having 
components which are three random variables with a joint probability
distribution supported in a solid half-cone or on its boundary surface. 

According to yet another result of \cite{stan-catrina}, stated below as
Proposition~\ref{omreq}, equation 
(\ref{nph}) is satisfied by the La\-place transforms $\,\varphi\,$ of all
suitably normalized nondegenerate $\,1$-Meix\-ner random vectors.

\section{Some details}\label{sd}
\setcounter{equation}{0}
Let us fix an integer $\,d\ge1$, a probability space
$\,(\varOmega,\mathcal{F}\nnh,P)$, and an $\,\bbR\hn^d\nh$-val\-ued 
random variable $\,\yy=(\yo\dots,\yd)$, assumed to have {\it finite moments
of all orders}, 
that is, to satisfy the condition
$\,\yi\in L\nnh^q(\varOmega,\mathcal{F}\nnh,P)\,$ for every   
$\,q\in(0,\infty)\,$ and all $\,i=1,\dots,d$. 
We are thus requiring $\,|\yi|^q\nh$, for each 
$\,q\hh$ and each component $\,\yi:\varOmega\to\bbR$, to be in\-te\-gra\-ble. 
Two vector subspaces of 
$\,L\nnh^2(\varOmega,\mathcal{F}\nnh,P)\,$ are of interest to us:
\begin{equation}\label{sbs}
\begin{array}{l}
\mathrm{the\ sub\-al\-ge\-bra\ }\hs F\nh\mathrm{\ generated\ by\
}\hs\yo,\hs\dots,\yd,\\
\mathrm{and\hs\ the\hs\ closure\hs\ }\,H\,\mathrm{\hs\ of\hs\ 
}\,F\mathrm{\hs\ in\hs\ }\,L\nnh^2(\varOmega,\mathcal{F}\nnh,P),
\end{array}
\end{equation}
the word `sub\-al\-ge\-bra' referring to the associative algebra of measurable
functions $\,\varOmega\to\bbR$, and `closure' to the 
$\,L\nnh^2$ topology. 
The space $\,F\,$ consists of all $\,f(\yo,\dots,\yd)$, with 
$\,f\,$ ranging over polynomial functions of $\,d\,$ real variables. 
For such $\,f\,$ having degrees at most $\,n$, the resulting
functions $\,f(\yo,\dots,\yd)\,$ form a vector sub\-space
$\,F\hskip-2.8pt_n\w$ of $\,F\nnh$, and we denote by $\,G\nh_n\w$ the 
$\,L\nnh^2\nnh$-or\-thog\-o\-nal complement of
$\,F\hskip-2.8pt_{n-1}\w$ in $\,F\hskip-2.8pt_n\w$, where we set 
$\,F\hskip-2.8pt_{-\nnh1}\w=G\nh_{-\nnh1}\w=\{0\}$.
Thus, $\,G\nh_0\w=F\hskip-2.8pt_0\w=\bbR$.

Clearly, $\,H\,$ is the Hil\-bert-space direct sum of the family
$\,\{\hn G\nh_n\nh\}_{n=0}^\infty$.

Some terminology: functions from $\,F$ are called {\it polynomial random
variables}, those among them lying in $\,G\nh_n\w$ are said to be 
{\it degree\/ $\,n\,$ homogeneous}, while each $\,G\nh_n\w$, and $\,H\nh$, is
referred to as the {\it $n$-th homogeneous chaos space\/} (and, respectively, 
the {\it chaos space}) {\it generated by\/} $\,\yo,\dots,\yd$.

We identify each $\,\yi$ with the operator $\,F \to F\hs$
of multiplication by $\,\yi$, that is,
\begin{equation}\label{mlt}
f(\yo,\dots,\yd)\,\mapsto\,\yi\hs f(\yo,\dots,\yd).
\end{equation}
As shown in \cite{accardi-nahni-02} and \cite{accardi-kuo-stan-04}, 
each operator (\ref{mlt}) sends each $\,G\nh_n\w$ into
$\,G\nh_{n-1}\w\hn\oplus G\nh_n\w\hn\oplus G\nh_{n+1}\w$.

For any integer $\,n\ge0$, we denote by $\,P\hskip-2.8pt_n\w:F\to G\nh_n\w$ 
the $\,L\nnh^2\nnh$-or\-thog\-o\-nal projection operator,
and by $I:F\to F\hs$ the identity operator. Thus,
$\,I\nh=\sum_{\hs n\ge0}\w P\hskip-2.8pt_n\w$, 
the sum being finite when applied to any vector in $\,F\nh$. The same
finiteness property is obviously shared by the operators
$\,a_i^0,\hs a_i^-:F\to F\nh$, where $\,i=1,\dots,d$, defined by
$\,a_i^0=\sum_{\hs n\ge0}\w P\hskip-2.8pt_n\w\yi P\hskip-2.8pt_n\w$, 
and $\,a_i^-=\sum_{\hs n\ge0}\w P\hskip-2.8pt_{n-1}\w\yi P\hskip-2.8pt_n\w$.
One calls
\begin{equation}\label{sao}
\ui\hs=\,a_i^-\hs+\,\displaystyle{\frac12}\,a_i^0
\end{equation}\
the $\,i\hh$th {\it sem\-i-an\-ni\-hi\-la\-tion\/} operator 
of (or, generated by) $\,\yo\dots,\yd$.

\section{1-Meix\-ner random vectors and equation {\rm(\ref{nph})}}\label{om}
\setcounter{equation}{0}
For $\,(\varOmega,\mathcal{F}\nnh,P)\,$ and $\,\yo\dots,\yd$ satisfying the
assumptions 
listed in the initial paragraph of Section~\ref{sd}, and $\,F\hs$ given by
(\ref{sbs}), let $\uo,\dots,\ud$ be the corresponding
joint sem\-i-an\-ni\-hi\-la\-tion operators, defined by (\ref{sao}). 
Following \cite[Definition 2]{stan-catrina}, we say that
$\yy=(\yo,\dots,\yd)\,$ is a $\,1${\it-Meix\-ner random vector\/} 
if, for all $\,i,j\in\{1,\dots,d\}$, the commutator $\,[\hs\ui,\yj]:F\to F$,
with $\,\yj$ as in (\ref{mlt}), is a linear combination of $\,\yo,\dots,\yd$ and $\,I\nh$, where $\,I\nh:F\to F\hs$
denotes the identity operator. In other words,
\begin{equation}\label{uix}
\left[\ui,\yj\right]
=\mathop{\displaystyle{\sum}}\limits_{k\hs=1}^d\alpha\hn_{i\hn jk}\w\yj+\beta_{ij}\w
I\mathrm{\ for\ all\ }\hs i,j\in\{1,\dots,d\}\hs\mathrm{\ and\ some\
}\,\alpha\hn_{i\hn jk}\w,\beta_{ij}\w\nh\in\bbR.
\end{equation}
A $\,1$-Meix\-ner random vector $\yy=(\yo,\dots,\yd)\,$ is called {\it nondegenerate\/}
if the multiplication operators $\yo,\dots,\yd\,$ and $\,I\hs$ are linearly independent.
In this case, as shown in \cite[Proposition 2]{stan-catrina}, 
the coefficients $\,\alpha\hn_{i\hn jk}\w$ in (\ref{uix}) are not only
(obviously) unique, buy also {\it totally symmetric in\/ } $\,i,j,k$.
Moreover, using an invertible af\-fine transformation, we are then
always free to assume the normalizing condition (\ref{nrm}).
The following conclusion was established in 
\cite[Lemma 6]{stan-catrina}.
\begin{prop}\label{omreq}
Let\/ $\yy=(\yo\dots,\yd)$ be a\/ $\,1$-Meix\-ner random vector 
with\/ {\rm(\ref{nrm})}.
Then equation\/ {\rm(\ref{nph})} is satisfied by the La\-place transform\/
$\,\varphi\,$ of\/ $\,\yy:\varOmega\to\bbR\hn^d\nh$, explicitly given by\/ 
$\,\varphi(x\hn_1\w,\dots,x\nh_d\w)
=\nnh\textstyle{\int}_{\hskip-3\varOmega}\w
e^{\lg x,\yy\hs\rg}\hh dP\nh$, 
where\/ $\,x=(x\hn_1\w,\dots,x\nh_d\w)\,$ and\/
$\,\lg x,\yy\hs\rg=x\hn_1\w\yo+\ldots+x\nh_d\w\yd$, 
for the data\/ $\,\ew,\alpha,\lr\,$ in\/ {\rm(\ref{tbg})} consisting of\/
$\,\ew\nh=\bbR\hn^d\nh$, its standard inner product\/ $\,\lr$, and\/
$\,\alpha\,$ with the components\/  
$\,\alpha\hn_{i\hn jk}\w=\lg\alpha(e\hn_i\w,e\nh_j\w),e\hn_k\w\rg$,  
in the standard or\-tho\-nor\-mal basis 
$\,e\hn_1\w,\dots,e\nh_d\w$, 
equal to the coefficients\/ $\,\alpha\hn_{i\hn jk}\w$ 
in\/ {\rm(\ref{uix})}.
\end{prop}
The assumption that $\,\yy\,$ is nondegenerate, made in 
\cite[Lemma 6]{stan-catrina}, may be dropped, as it follows from (\ref{nrm}).
Namely, if $\,I\nh:F\to F\hs$ were a linear combination
of the multiplication operators in (\ref{mlt}), the equality 
$\,I=s\hn_1\w\yo+\ldots+s\hn_d\w\yd$ applied to each $\,\yi$ 
would, via (\ref{nrm}) combined with integration, give $\,s\hn_i\w=0$.

\section*{Appendix. The inverse Laplace transforms for {\rm(\ref{cmp})} with
$\,d\le3$}
\setcounter{section}{10}
\setcounter{equation}{0}
For the reader's convenience, we describe here the inverse La\-place
transforms, appearing in \cite{stan-catrina},
of the functions $\hs\varphi\hs$ with (\ref{fcs}) -- (\ref{cmp}) 
and $\,\varphi(0)=1$.

If $\,d=1$, as $\,\varphi:(-\infty,1)\to\bbR\,$ is given by 
$\,\varphi(x)=e^{-x}\nnh/(1-x)$, from (\ref{lap}),
\begin{equation}\label{one}
\varphi\,\mathrm{\ is,\ obviously,\ the\ La\-place\ transform\ of\
}\,y\mapsto e^{-y-1}\hn u(y+1),
\end{equation}
$u\,$ being the unit-step (Heaviside) function. Next,
as one sees using the new variable $\,\sqrt{1-x\hs}\sqrt{y+1/2\hs}\,$ and the
Gauss\-i\-an integral
$\,\int_{\nh-\nh\infty}^\infty e^{-\xi^2}\nh d\xi=\sqrt{\pi\hs}$,
if $\,x\in(-\infty,1)$,
\begin{equation}\label{mhf}
\displaystyle{\int}_{\hskip-5pt-\nnh1\nh/2}^\infty\,
\displaystyle{\frac{e^{(x-1)y}}{\sqrt{y+1/2}}}\,dy\,
=\,\sqrt{\pi\hs}\,\displaystyle{\frac{e^{(1-x)/2}}{\sqrt{1-x}}}\hs.
\end{equation}
Consequently, the function $\,\varphi\,$ in (\ref{cmp}), given,
for $\,d=2$, by
\begin{equation}\label{fxx}
\begin{array}{l}
\varphi(x,\xi)\,\hs
=\hs\hs\,e^{-\xi}\hn[(1-\xi)^2\nnh-x^2]^{-\nnh1/2}\mathrm{\nh,\hs\ equals\hs\
the\hs\ La\-place\hs\ transform\hs\ of}\\
\rto\nh\nnh\ni\hn\nh
(y,\eta)\nh\mapsto\nh(\pi e)^{-\nnh1}\hn e^{-\eta}\hh
[(1+\eta)^2\nnh-y^2]^{-\nh1/2}\hh u(1+\eta+y)u(1+\eta-y).
\end{array}
\end{equation}
In fact, in terms of the new variables
$\,\lambda=(\eta+y)/2\,$ and $\,\vt=(\eta-y)/2$, with the Ja\-cob\-i\-an
determinant
$\,\partial\hs(y,\eta)/\partial\hh(\lambda,\vt)=2$, and
$\,xy+\xi\eta=(\xi+x)\lambda+(\xi-x)\vt$, 
the integral of the second line of (\ref{fxx}) 
multiplied by $\,e\hh^{xy+\xi\eta}\hs dy\,d\eta\,$ has separated
variables, and hence equals
$\,(\pi e)^{-\nnh1}$ times the product of two integrals, namely,  of 
$\,e^{(\xi-1+x)\lambda}\hn(\lambda+1/2)^{-\nh1/2}\hh u(\lambda+1/2)\,
d\lambda\,$ and $\,e^{(\xi-1-x)\evt}\hn(\vt+1/2)^{-\nh1/2}\hh u(\vt+1/2)\,
d\vt\,$ which, by
(\ref{mhf}) with $\,(x,y)\,$ replaced by $\,(\xi+x,\lambda)\,$ or
$\,(\xi-x,\vt)$, are equal, respectively, to
$\,\sqrt{\pi\hs}\hs e^{(1-\xi\mp x)/2}\hn(1-\xi\mp x)^{-\nh1/2}\nnh$. Thus, 
their product multiplied by $\,(\pi e)^{-\nnh1}$ coincides with 
the right-hand side in the first line of (\ref{fxx}). 

For $\,d=3\,$ the inverse La\-place transform of 
the function $\,\varphi\,$ in (\ref{cmp}) is {\it not\/} a function, but
rather a measure supported in a half of a Lo\-rentz\-i\-an light cone 
\cite[Proposition 7]{stan-catrina}, as described below.

Generally, given a $\,d$-di\-men\-sion\-al oriented, time-orient\-ed
Lo\-rentz\-i\-an vector space $\,\ew$
with the $\,-\nnh+\ldots+\,$ inner product $\,(\,,\hh)\,$ and
the volume form $\,\varTheta\nh$, the future light cone
$\,\partial\hh C\hskip-2pt_+\w\nnh\subseteq\ew\,$ is well known
\cite[p.\,62]{kolk-varadarajan} to carry a canonical smooth measure $\,\varLambda$,
invariant under the iso\-chronous Lo\-rentz group. Explicitly: 
$\,\varLambda_y\w$, as an exterior $\,(d-1)$-form,
at any  $\,y\in\partial\hh C\hskip-2pt_+\w$, equals the
restriction to $\,y^\perp$ of $\,\varTheta(y'\nnh,\,\cdot\,,\dots,\,\cdot\,)$,
where the Lo\-rentz\-i\-an or\-thog\-o\-nal complement $\,y^\perp$ obviously
coincides with the tangent space of $\,\partial\hh C\hskip-2pt_+\w$ at
$\,y$, and $\,y'\in\ew$ is any vector such that $\,(y,y')=-\nnh1$. For obvious
dimensional reasons, due to uniqueness of $\,y'$ up to adding vectors from
$\,y^\perp\nnh$, this definition is correct in the sense that
$\,\varLambda_y\w$ does not depend on the choice of $\,y'\nh$.

The push-for\-ward of $\,\varLambda\,$ under the or\-thog\-o\-nal projection
$\,\partial\hh C\hskip-2pt_+\w\nnh\to W\smallsetminus\{0\}\,$ onto any
spacelike hyperplane $\,W\nh\subseteq\ew\,$ equals $\,dw/|w|$, in terms
of the variable $\,w\in W\nh$, which in turn becomes
\begin{equation}\label{tht}
\varLambda\,=\,r\hh^{d-3}\hs dr\,d\zeta
\end{equation}
if one further identifies $\,W\hn$ with $\,\bbR\hn^{d-1}$, letting 
$\,r\,$ and $\,\zeta$ range, respectively, over $\,[\hs0,\infty)\,$ and
the unit sphere $\,S\hh^{d-2}$ centered at $\,0\,$ in $\,\bbR\hn^{d-1}\nnh$.

For our purposes $\,\ew\nh=\bbR\hn^d$ and the future light cone 
$\,\partial\hh C\hskip-2pt_+$ is given by $\,1+\eta>r$ with 
$\,\eta=y\nh_d\w$ and $\,r\ge0\,$ such that 
$\,r^2\nnh=\textstyle{\sum}\hs_{i\hh<\hn d}\w\,y_i^2$,
where $\,y=(y\hn_1\w,\dots,y\hn_d\w)$. Thus, $\,(0,\dots,0,-\nnh1)\,$ serves
as the origin (vertex of the cone). The measure $\,\varLambda\,$ on
$\,\partial\hh C\hskip-2pt_+$ is, by (\ref{tht}), the push-for\-ward, under
the mapping $\,S\hh^{d-2}\nh\times(0,\infty)
\to\partial\hh C\hskip-2pt_+\w\subseteq\bbR\hn^d$ which sends 
$\,(\zeta,r)\,$ to $\,(r\zeta,\hh r-1)$, 
of the standard product measure on $\,S\hh^{d-2}\nh\times(0,\infty)$
multiplied by $\,r\hh^{d-3}\nh$. When $\,d=3$, in the notation of (\ref{inv}), 
\begin{equation}\label{ltf}
\begin{array}{l}
\mathrm{our\hs\ }\,\,\varphi\,\hs
=\,\,e^{-\xi}\hh[(1-\xi)^2\nnh-s^2]^{-\nnh1/2}\,\,\mathrm{\ 
is\ the\ La\-place\ trans}\hyp\\
\mathrm{form\ of\ the\ product\ of\ }\hh\varLambda\mathrm{\ and\ the\
function\ }\hh(2\pi e)^{-\nnh1}\nh e^{-\eta}\nh.
\end{array}
\end{equation}
We will derive (\ref{ltf}) from the following fact, 
verified in the next paragraph: 
whenever $\,\xi,s\in\bbR\,$ and $\,1-\xi>s\ge0$, so that
$\,1-\xi-s\cos\theta>0$,
\begin{equation}\label{rsd}
\displaystyle{\int}_{\hskip-5pt0}^\pi
\displaystyle{\frac{d\theta}{1-\xi-s\cos\theta}}\,
=\,\pi\hs[(1-\xi)^2\nnh-s^2]^{-\nnh1/2}\nh.
\end{equation}
Namely, by  continuity, we are free to assume that $\,s>0$. 
The left-hand side of (\ref{rsd}) equals one-half of the analogous 
integral over $\,[\hs0,2\pi]$, and so it coincides with $\,1/2$ times 
the contour integral over the circle $\,|z|=1\,$ in $\,\bbC\,$ of the 
hol\-o\-mor\-phic  $\,1$-form $\,f(z)\,dz$, where
$\,f(z)=(iz)^{-\nnh1}\nh/[1-\xi-s\hh(\nh z+\nnh z^{-\nnh1}\nh)/2]
=2i/[s\hn z^2\nnh+2(\xi-1)z+s\hh]$.
The polynomial $\,z\mapsto s\hn z^2\nnh+2(\xi-1)z+s\,$ 
has the roots 
$\,z\hn_\pm\w\nh=\nh(1-\xi\pm[(1-\xi)^2\nnh-s^2]^{1/2})\nh/\nnh s$
and $\,0<z_-\w\nnh<1<z_+\w$. In fact, $\,z_+\w>1\,$ as $\,1-\xi>s$, 
while $\,0<z_-\w\nnh<1\,$ since $\,z_+\w z_-\w=s/s=1$. 
Consequently, the integral in (\ref{rsd}) equals
$\,\pi i\,$ times the residue of $\,f(z)\,$ at $\,z=z_-\w$. The residue, being 
clearly the limit of $\,2i(z-z_-\w)/[sz^2\nnh+2(\xi-1)z+s]$ as 
$\,z\to z_-\w$, is thus $\,2i/[s(z_-\w\nh-z_+\w)]
=-i[(1-\xi)^2\nnh-s^2]^{-\nnh1/2}\nh$, and (\ref{rsd}) follows.

Denoting by $\,\sm$ the 
$\,m$-di\-men\-sion\-al volume of the unit $\,m$-sphere $\,S^m$ 
centered at $\,0\,$ in $\,\bbR\hn^{m+1}\nnh$, we see
that, for any 
fixed vector $\,\chi\in\bbR\hn^{m+1}$ with $\,|\chi|=t$, the integral over
$\,S^m$ of the function $\,\zeta\mapsto e^{\lg\chi,\zeta\rg}$ equals
$\,\smo\hh Q\nh_m\w(t)$, where
\begin{equation}\label{qmp}
Q\nh_m\w(t)\,=
\nnh\displaystyle{\int}_{\hskip-5pt0}^\pi e\hh^{t\cos\theta}\nh
\sin^{m-1}\nnh\nh\theta\,d\theta\mathrm{,\ \ and\ so\ \
}\,\smo Q\nh_m\w(0)=\sm,
\end{equation}
while $\,\smo=2\pi^{m/2}\nh/\hh\Gamma(m/2)$. 
Given $\,x,y\in\bbR\hn^d\nh$, let us now write $\,x=(v,\xi)\,$ and
$\,y=(w,\eta)$,
where $\,v,w\in\bbR\hn^{d-1}$ and $\,\xi,\eta\in\bbR$, as well as 
$\,s=|v|\,$ and $\,r=|w|$.

In any dimension $\,d\ge3$,  
the La\-place transform of the measure
$\,(2\pi e)^{-\nnh1}\nh e^{-\eta}\hh\varLambda$, mentioned in (\ref{ltf})
for $\,d=3$, 
depends on $\,x\in\bbR\hn^d$ only through $\,\xi\,$ and $\,s\hh$: due to 
(\ref{tht}), 
with $\,v=r\zeta$, where $\,\zeta\in S^{d-2}\nh$, as $\,\eta=r-1$ on
$\,\partial\hh C\hskip-2pt_+$, the
transform
\[
\mathrm{sends\ }\,x=(v,\xi)\,\mathrm{\ to
}\,\,(2\pi e)^{-\nnh1}\nh\nnh\displaystyle{\int}_{\hskip-5pt0}^\infty
\hh\displaystyle{\int}_{\hskip-5.5ptS^{d-2}}\w
e^{r\lg v,\zeta\rg}\hs e\hh^{(\xi-1)(r-1)}\nh r\hh^{d-3}\hs d\zeta\,dr,
\]
which amounts to
$\,(s,\xi)\,\mapsto(2\pi e)^{-\nnh1}\hh\sdt
\hh\textstyle{\int}_{\hskip-1.8pt0}^\infty
Q\nh_{d-2}\w(r\nh s)\,e\hh^{(\xi-1)(r-1)}r^{d-3}\nh d\hh r\,$ 
for $\,Q\nh_{d-2}\w$ and $\,\sdt$ as in (\ref{qmp}) 
and, using (\ref{qmp}), we rewrite it as
\[
(s,\xi)\,\mapsto\,(2\pi)^{-\nnh1}\hn e^{-\xi}\hn\sdt
\nnh\displaystyle{\int}_{\hskip-5pt0}^\pi
\displaystyle{\int}_{\hskip-5pt0}^\infty
e^{r(\xi-1+s\cos\theta)}\hs r\hh^{d-3}\hn \sin^{d-3}\nnh\theta\,d\hh r\,
d\theta.
\]
For $\,d=3$, with $\,\sz=2$, this becomes
\[
(s,\xi)\,
\mapsto\,\pi^{-\nnh1}\hn e^{-\xi}\nnh\displaystyle{\int}_{\hskip-5pt0}^\pi
\displaystyle{\int}_{\hskip-5pt0}^\infty
e^{-r(1-\xi-s\cos\theta)}\nh dr\,d\theta.
\]
Integrating over $\,r\,$ first, with $\,1-\xi-s\cos\theta>0$, we obtain
\[
(s,\xi)\,\mapsto
\pi^{-\nnh1}\hn e^{-\xi}\nnh\displaystyle{\int}_{\hskip-5pt0}^\pi
\displaystyle{\frac{d\theta}{1-\xi-s\cos\theta}}\hs.
\]
By (\ref{rsd}), the right-hand side equals 
$\,e^{-\xi}\hh[(1-\xi)^2\nnh-s^2]^{-\nnh1/2}\nh$, proving (\ref{ltf}).

\ \

\noindent{\bf Data Availability.} 
All data generated or analyzed during this study
are included in this published article.

\vskip2pt
\noindent{\bf Conflict of interest.} 
The authors have no relevant financial or
non-financial interests to disclose.


\end{document}